\documentclass[11pt,reqno]{amsart}
\usepackage{tikz}
\usepackage[a4paper,margin=1.03in]{geometry}
\usepackage{amsmath,amssymb,amsthm,mathtools}
\usepackage{enumitem}
\usepackage{booktabs}
\usepackage{microtype}
\usepackage{hyperref}
\usepackage[nameinlink,capitalize]{cleveref}
\usepackage{setspace}
\numberwithin{equation}{section}
\hypersetup{
  colorlinks=true,
  linkcolor=blue,
  citecolor=blue,
  urlcolor=blue,
  pdftitle={Self-convolution can destroy local univalence for bounded convex harmonic mappings},
  pdfauthor={Deguang Zhong}
}

\usepackage{xcolor}

\newtheorem{theorem}{Theorem}[section]
\newtheorem{proposition}[theorem]{Proposition}
\newtheorem{lemma}[theorem]{Lemma}
\newtheorem{corollary}[theorem]{Corollary}

\theoremstyle{definition}

\newtheorem{problem}[theorem]{Problem}
\theoremstyle{remark}
\newtheorem{remark}[theorem]{Remark}

\newcommand{\D}{\mathbb D}
\newcommand{\T}{\mathbb T}
\newcommand{\C}{\mathbb C}
\newcommand{\R}{\mathbb R}
\newcommand{\KH}{\mathcal K_H}
\newcommand{\KHZ}{\mathcal K_H^0}
\newcommand{\SH}{\mathcal S_H}
\newcommand{\SHZ}{\mathcal S_H^0}
\newcommand{\convh}{\mathbin{\widetilde *}}
\newcommand{\Jac}{J}
\newcommand{\Poisson}{\mathcal P}

\newcommand{\e}{\mathrm e}

\title[A counterexample to an open problem of Dorff]{A counterexample to an open problem of Dorff}

\author[Zhi-Gang Wang and Deguang Zhong
]{Zhi-Gang Wang$^*$ and Deguang Zhong  
}

\address{\noindent Zhi-Gang Wang  \vskip.05in
School of Mathematics and Statistics, Hunan First Normal University, Changsha 410205, Hunan, P. R.
China.}
\email{\textcolor[rgb]{0.00,0.00,0.84}{sjyzhigangwang$@$hnfnu.edu.cn}}

\address{\noindent Deguang Zhong\vskip.05in
Institute of Applied Mathematics, Shenzhen Polytechnic University,
Shenzhen 518055, Guangdong,   P. R.
China.}
\email{\textcolor[rgb]{0.00,0.00,0.84}{zhongdg1014$@$szpu.edu.cn}}

\begin{document}

\begin{abstract}
The classical Pólya-Schoenberg conjecture, proved by Ruscheweyh-Sheil-Small, asserts that the convolution of two normalized convex univalent functions is again convex. This property fails to carry over to planar harmonic mappings. In 2001, Dorff posed the open problem whether the self-convolution of a normalized convex harmonic mapping with bounded image must remain in the same class. We construct a normalized sense-preserving harmonic diffeomorphism that maps the unit disk onto an ellipse; its self‑convolution has vanishing Jacobian at some interior point of the unit disk, which provides a negative answer to Dorff’s open problem.
\end{abstract}

\thanks{$^*$Corresponding author.}
\maketitle

\section{Introduction and statement of the main result}

For two normalized analytic functions on the unit disk \(\mathbb D\),
\[
  \varphi(z)=z+\sum_{n=2}^{\infty}c_nz^n,
  \quad
  \psi(z)=z+\sum_{n=2}^{\infty}d_nz^n,
\] 
their Hadamard product (or convolution) is defined by
\[
  (\varphi*\psi)(z)=z+\sum_{n=2}^{\infty}c_nd_nz^n.
\]
A classical result of Ruscheweyh-Sheil-Small \cite{RuscheweyhSheilSmall1973}, settling the P\'olya-Schoenberg conjecture, states in particular that the class of normalized analytic convex functions is closed under convolution.  This closure phenomenon is a central structural feature of the analytic function theory.
However, the harmonic situation is subtler. A complex-valued harmonic mapping of \(\mathbb D\) can be written as
\(
 f=h+\overline g,
\)
where $h$ and $g$ are analytic functions.  

If
\[
 f=h+\overline g,
 \quad
 F=H+\overline G,
\]
then their harmonic convolution is defined by
\[
  f\convh F=h*H+\overline{g*G}.
\]
The systematic study of harmonic univalent mappings began with Clunie-Sheil-Small \cite{ClunieSheilSmall1984}; see also the monograph of Duren \cite{Duren2004}.  In contrast with the analytic case, harmonic convolution need not preserve univalence, local univalence, or convexity.  In 2001, Dorff \cite{Dorff2001} developed a number of positive convolution theorems for right half-plane and strip mappings, under local-univalence hypotheses, and isolated the following natural bounded-domain problem.

\begin{problem}
\cite[Question 10]{Dorff2001}\label{q:Dorff}
{\it If $f\in\KH$ maps $\D$ onto a bounded domain, is
\(
  f\convh f\in\KH?
\)}
\end{problem}

In Dorff's notation, $\KH$ denotes the normalized sense-preserving univalent harmonic mappings of $\D$ onto convex domains, and let \(\KH^0\subset \KH\) be the subclass with the additional normalization \(g'(0)=0\).  His problem arose because the preceding results in \cite{Dorff2001} concern unbounded domains, whereas a particular bounded polygonal example has a well-behaved self-convolution.

Subsequent work has greatly refined the local-univalence and directional-convexity theory of harmonic convolutions, especially for half-plane and strip mappings; see, for example, Dorff-Nowak-Wo\l oszkiewicz \cite{DorffNowakWoloszkiewicz2012}, Li-Ponnusamy \cite{LiPonnusamy2013,LiPonnusamy2019,LiPonnusamy2022}, Jiang-Rasila-Sun \cite{SunJiangRasila2015}, and Ahmad El-Faqeer-Ng-Supramaniam \cite{ElFaqeerEtAl2021}.  

Our main result is as follows.

\begin{theorem}\label{thm:main}
There exists $f\in\KHZ$ such that $f(\D)$ is an ellipse and
\(
 f\convh f
\)
is not locally univalent in $\D$. 
\end{theorem}

\begin{remark}\label{rem:strong-failure}
{\rm Theorem \ref{thm:main} provides a negative answer to Problem \ref{q:Dorff}.}
\end{remark}


This paper is organized as follows.
In Section \ref{sec:prelim}, we fix notation and recall the harmonic Jacobian criterion, harmonic convolution, and the Rad\'o-Kneser-Choquet theorem.  Section \ref{sec:boundary} proves that the rational boundary map is a circle homeomorphism.  Section \ref{sec:fourier} computes its Fourier coefficients exactly. Section \ref{sec:normalization} performs the affine normalization and identifies the image ellipse. 
Section \ref{sec:coefficients} processes the coefficient structure after normalization. 
Section \ref{sec:convolution} derives closed forms for the self-convolution derivatives.  Section \ref{sec:sign} gives the decisive Jacobian sign change and proves Theorem \ref{thm:main}. 

\vskip .10in
\section{Preliminaries}\label{sec:prelim}
To prove our main result, we require the following preliminaries.
\subsection{Harmonic mappings and normalization}
Let $f:\D\to\C$ be harmonic.  Since $\D$ is simply connected, there exist analytic functions $h$ and $g$ on $\D$ such that
\begin{equation}\label{eq:harmonic-decomp}
 f=h+\overline g.
\end{equation}
The decomposition is unique after fixing $g(0)=0$.  In complex notation,
\[
 f_z=h',\quad f_{\bar z}=\overline{g'},
\]
and the Jacobian is
\begin{equation}\label{eq:jacobian}
 \Jac_f=|h'|^2-|g'|^2.
\end{equation}
Thus $f$ is sense-preserving wherever $|h'|>|g'|$.  The quotient
\[
 \omega=\frac{g'}{h'}
\]
is the second complex dilatation of $f$ whenever $h'\neq0$.

Following the standard notation originating in \cite{ClunieSheilSmall1984}, let $\SH$ be the family of normalized sense-preserving univalent harmonic mappings $f=h+\overline g$ of $\D$ satisfying
\(
 f(0)=0\) and \(h'(0)=1.
\)
Let $\SHZ$ denote the subclass for which additionally $g'(0)=0$.  

We recall the following classical result due to Lewy \cite{Lewy1936}.

\begin{theorem}\label{thm:Lewy}
If a planar harmonic mapping is locally one-to-one at a point, then its Jacobian does not vanish at that point.
\end{theorem}

Combined with the inverse function theorem, this means that for harmonic mappings the nonvanishing Jacobian is exactly the local-diffeomorphism condition.  In the sense-preserving normalized class, one requires $\Jac_f>0$ throughout $\D$.

\subsection{Harmonic convolution}
Suppose
\[
 f=h+\overline g,
 \quad
 F=H+\overline G,
\]
where
\[
 h(z)=z+\sum_{n=2}^\infty a_nz^n,
 \quad
 g(z)=\sum_{n=1}^\infty b_nz^n,
\]
and similarly
\[
 H(z)=z+\sum_{n=2}^\infty A_nz^n,
 \quad
 G(z)=\sum_{n=1}^\infty B_nz^n.
\]
Their harmonic convolution is
\begin{equation}\label{eq:harmonic-convolution}
 (f\convh F)(z)
 = (h*H)(z)+\overline{(g*G)(z)},
\end{equation}
where
\[
 (h*H)(z)=z+\sum_{n=2}^\infty a_nA_nz^n,
 \quad
 (g*G)(z)=\sum_{n=1}^\infty b_nB_nz^n.
\]

For the self-convolution $F=f\convh f$, if
\[
 f(z)=z+\sum_{n=2}^\infty M_nz^n
 +\overline{\sum_{n=2}^\infty N_nz^n},
\]
then
\begin{equation}\label{eq:selfconv-series}
 F(z)=z+\sum_{n=2}^\infty M_n^2z^n
 +\overline{\sum_{n=2}^\infty N_n^2z^n}.
\end{equation}

\subsection{The Rad\'o-Kneser-Choquet theorem}
The geometric input of our construction is the classical Rad\'o-Kneser-Choquet theorem.  Rad\'o formulated the underlying problem in 1926, Kneser \cite{Kneser1926} supplied a proof, and Choquet \cite{Choquet1945} later rediscovered and generalized the result; see Duren \cite[Chapter 3]{Duren2004} for a modern account.

\begin{theorem}\label{thm:RKC}
Let $\Omega\subset\C$ be a bounded convex Jordan domain and let
\(
 \gamma:\T\to\partial\Omega
\)
be an orientation-preserving homeomorphism.  Then the Poisson extension $\Poisson[\gamma]$ is a harmonic diffeomorphism of $\D$ onto $\Omega$.
\end{theorem}

We shall only need $\Omega=\D$.  The strength of Theorem \ref{thm:RKC} is that no smallness condition on the boundary distortion is required: topological monotonicity on the boundary together with convexity of the image is sufficient.

\subsection{Real-affine postcomposition}
A real-linear map of the plane can be written as
\begin{equation}\label{eq:real-linear}
 T(w)=\alpha w+\beta\overline w,
\end{equation}
where $\alpha,\beta\in\C$.  Its Jacobian is
\[
 \Jac_T=|\alpha|^2-|\beta|^2.
\]
If this is positive, then $T$ is an orientation-preserving real-linear automorphism.  Postcomposition of a harmonic diffeomorphism by such a map preserves harmonicity, injectivity and orientation, while turning circles into ellipses.  This simple observation will give bounded convexity for free after the disk self-map has been constructed.
\vskip .10in
\section{A rational circle homeomorphism}\label{sec:boundary}

For $r\in(-1,1)$, set
\begin{equation}\label{eq:Blaschke}
 B_r(z)=\frac{z-r}{1-rz}.
\end{equation}
Note that $|B_r|=1$ on $\T$.  We use the particular parameters
\begin{equation}\label{eq:ab}
 a=-\frac{10}{11},\quad b=\frac1{44},
\end{equation}
and define
\begin{equation}\label{eq:Gamma}
 \Gamma(z)=z\frac{B_a(z)}{B_b(z)}
 =\frac{z(z-a)(1-bz)}{(1-az)(z-b)}
 \quad (|z|=1).
\end{equation}
Since $B_a$ and $B_b$ are unimodular on $\T$, it implies $|\Gamma|=1$ .

The map is not the boundary trace of an analytic disk self-map because of the quotient by $B_b$; that is precisely what supplies negative Fourier modes.  Nevertheless its restriction to the circle is a very regular degree-one homeomorphism.

\begin{lemma}[Angular derivative of a real Blaschke factor]\label{lem:angular-Blaschke}
Let $r\in(-1,1)$ and write $z=\e^{it}$.  Then
\[
 \frac{d}{dt}\arg B_r(\e^{it})
 =\frac{1-r^2}{1-2r\cos t+r^2}.
\]
\end{lemma}

\begin{proof}
The logarithmic derivative gives
\[
 \frac{d}{dt}\log B_r(\e^{it})
 = i\e^{it}\frac{B_r'(\e^{it})}{B_r(\e^{it})}.
\]
Its imaginary part is the derivative of the argument.  Since
\[
 B_r'(z)=\frac{1-r^2}{(1-rz)^2},
\]
a direct simplification on $|z|=1$ yields the displayed Poisson-kernel expression.
\end{proof}

\begin{proposition}\label{prop:Gamma-homeo}
The map $\Gamma:\T\to\T$ defined in \eqref{eq:Gamma} is an orientation-preserving $C^\infty$ diffeomorphism.
\end{proposition}

\begin{proof}
Let $z=\e^{it}$ and put $x=\cos t$.  By Lemma \ref{lem:angular-Blaschke}, we have
\begin{align}
 \frac{d}{dt}\arg\Gamma(\e^{it})
 &=1+\frac{1-a^2}{1-2ax+a^2}
   -\frac{1-b^2}{1-2bx+b^2}.
\end{align}
Substituting \eqref{eq:ab} and reducing to a common denominator gives the exact identity
\begin{equation}\label{eq:angular-rational}
 \frac{d}{dt}\arg\Gamma(\e^{it})
 =\frac{19360x^2+20856x-41119}
 {(88x-1937)(220x+221)}.
\end{equation}
For $-1\le x\le1$, we have
\[
 88x-1937<0,
 \quad
 220x+221>0,
\]
so the denominator in \eqref{eq:angular-rational} is negative.  The numerator
\[
 N(x)=19360x^2+20856x-41119
\]
is a convex quadratic.  Hence its maximum over $[-1,1]$ is attained at an endpoint, and
\[
 N(1)=-903<0,
 \quad
 N(-1)=-42615<0.
\]
Thus the numerator is negative throughout $[-1,1]$, and consequently
\begin{equation}\label{eq:angular-positive}
 \frac{d}{dt}\arg\Gamma(\e^{it})>0
 \quad(t\in\R).
\end{equation}

It remains only to identify the degree.  The factor $z$ has degree $1$ on $\T$, each $B_r$ has degree $1$, and therefore
\[
 \deg\Gamma=1+1-1=1.
\]
A smooth circle map of degree one with strictly positive angular derivative has a strictly increasing lift $\theta:\R\to\R$ satisfying $\theta(t+2\pi)=\theta(t)+2\pi$.  Hence it is an orientation-preserving diffeomorphism of the circle.
\end{proof}

\begin{remark}\label{rem:not-delicate}
{\rm The sign verification above is deliberately organized without locating any root of $N$.  The endpoint test is enough because $N$ is convex.  This elementary observation is useful later in the stability argument: strict positivity of the angular derivative is a compact, open condition in the two real parameters.}
\end{remark}

Let
\begin{equation}\label{eq:U}
 U=\Poisson[\Gamma]
\end{equation}
be the Poisson extension of the boundary values.  We immediately obtain the geometric half of the construction.

\begin{corollary}\label{cor:U-diffeo}
The map $U$ is a sense-preserving harmonic diffeomorphism of $\D$ onto $\D$.
\end{corollary}

\begin{proof}
Applying Theorem \ref{thm:RKC} with $\Omega=\D$ and using Proposition \ref{prop:Gamma-homeo},
we get the desired assertion.
\end{proof}
\vskip .10in
\section{Exact Fourier expansion}\label{sec:fourier}

We now exploit the special rational form of $\Gamma$.  Since the parameters are real, all Fourier coefficients below are real.

Substituting \eqref{eq:ab} into \eqref{eq:Gamma}, we have
\begin{equation}\label{eq:Gamma-rational-explicit}
 \Gamma(z)
 =-\frac{z(z-44)(11z+10)}{(10z+11)(44z-1)}.
\end{equation}
A partial-fraction decomposition gives
\begin{equation}\label{eq:partial-fraction}
 \Gamma(z)
 =-\frac z{40}+\frac{9717}{8800}
 +\frac{79335}{86944(44z-1)}
 -\frac{104181}{49400(10z+11)}.
\end{equation}
The annulus
\[
 |b|<|z|<\frac1{|a|}
\]
contains $\T$.  On this annulus, the last two terms in \eqref{eq:partial-fraction} expand in opposite directions:
\begin{align}
 \frac{1}{44z-1}
 &=\frac{1}{44z}\frac1{1-(44z)^{-1}}
 =\frac1{44}\sum_{k=0}^\infty b^kz^{-k-1},\label{eq:negative-expand}\\
 \frac1{10z+11}
 &=\frac1{11}\frac1{1+(10/11)z}
 =\frac1{11}\sum_{k=0}^\infty a^kz^k.\label{eq:positive-expand}
\end{align}
Thus on the unit circle \(\mathbb T\), we obtain an absolutely convergent Laurent series
\begin{equation}\label{eq:Laurent}
 \Gamma(z)=\sum_{n=-\infty}^{\infty}c_nz^n.
\end{equation}

\begin{proposition}[Fourier coefficients]\label{prop:Fourier}
The coefficients in \eqref{eq:Laurent} satisfy
\begin{align}
 c_0&=\frac{79335}{86944},\label{eq:c0}\\
 c_1&=\frac{295}{1976},\label{eq:c1}\\
 c_2&=-\frac{861}{5434},\label{eq:c2}\\
 c_{-1}&=\frac{79335}{3825536},\label{eq:cm1}
\end{align}
and, for every $n\ge2$,
\begin{equation}\label{eq:geomcoeff}
 c_n=c_2a^{n-2},
 \quad
 c_{-n}=c_{-1}b^{n-1}.
\end{equation}
\end{proposition}

\begin{proof}
It follows from \eqref{eq:negative-expand} that
\[
 \frac{79335}{86944(44z-1)}
 =\frac{79335}{3825536}\sum_{k=0}^\infty b^kz^{-k-1},
\]
which yields \eqref{eq:cm1} and the negative-frequency relation in \eqref{eq:geomcoeff}.

For the positive frequencies, write
\[
 C^*=-\frac{104181}{543400}.
\]
Then \eqref{eq:positive-expand} contributes $C^*a^nz^n$ for $n\ge0$, while the term $-z/40$ modifies only the coefficient of $z$.  Hence
\[
 c_1=C^*a-\frac1{40}=\frac{295}{1976},
\]
\[
 c_2=C^*a^2=-\frac{861}{5434},
\]
and for $n\ge2$,
\[
 c_n=C^*a^n=c_2a^{n-2}.
\]
Combining the constant part of \eqref{eq:partial-fraction} with $C^*$ gives \eqref{eq:c0}.
\end{proof}

Since $U$ is the Poisson extension of the boundary series, its harmonic expansion is
\begin{equation}\label{eq:U-series}
 U(z)=c_0+\sum_{n=1}^\infty c_nz^n
 +\sum_{n=1}^\infty c_{-n}\overline z^{\,n}.
\end{equation}
The geometric decay in \eqref{eq:geomcoeff} gives absolute and locally uniform convergence together with all differentiated series.

\begin{remark}\label{rem:why-quotient}
{\rm If one used only an analytic finite Blaschke product as boundary data, its negative Fourier coefficients would vanish and the Poisson extension would remain analytic.  The analytic convex class is closed under self-convolution, so this could not produce a counterexample.  The quotient $B_a/B_b$ creates a controlled negative-frequency tail, while the angular derivative can still be kept positive.  The two tails decay at different geometric rates $|a|$ and $|b|$, which later makes the coefficient squares tractable.}
\end{remark}
\vskip .10in
\section{Affine normalization and the image ellipse}\label{sec:normalization}

The disk diffeomorphism $U$ is not normalized in the sense of $\SHZ$, but a real-affine postcomposition fixes this exactly.

Set
\begin{equation}\label{eq:Delta}
 \Delta=c_1^2-c_{-1}^2.
\end{equation}
Using \eqref{eq:c1} and \eqref{eq:cm1}, we obtain
\begin{equation}\label{eq:Delta-exact}
 \Delta=
 \frac{24606462475}{1125748129792}>0.
\end{equation}
Define
\begin{equation}\label{eq:T}
 T(w)=\frac{c_1w-c_{-1}\overline w}{\Delta}
\end{equation}
and
\begin{equation}\label{eq:f-def}
 f(z)=T\bigl(U(z)-c_0\bigr).
\end{equation}

\begin{proposition}\label{prop:normalization}
The map $T$ is an orientation-preserving real-linear automorphism of $\C$, and $f$ belongs to $\KHZ$.
\end{proposition}

\begin{proof}
The complex-linear and anti-linear coefficients of $T$ are $c_1/\Delta$ and $-c_{-1}/\Delta$, respectively.  It follows that
\[
 \Jac_T
 =\frac{c_1^2-c_{-1}^2}{\Delta^2}
 =\frac1\Delta>0.
\]
Thus $T$ is an orientation-preserving real-linear automorphism.  By Corollary \ref{cor:U-diffeo}, $U-c_0$ is a harmonic diffeomorphism from $\D$ onto the translated disk $\D-c_0$, and hence $f$ is a sense-preserving univalent harmonic mapping onto a convex domain.

To check normalization, write
\[
 U-c_0=H_0+\overline{G_0},
\]
where
\[
 H_0(z)=\sum_{n=1}^{\infty}c_nz^n,
 \quad
 G_0(z)=\sum_{n=1}^{\infty}c_{-n}z^n.
\]
Then
\begin{equation}\label{eq:hg-formulas}
 h=\frac{c_1H_0-c_{-1}G_0}{\Delta},
 \quad
 g=\frac{c_1G_0-c_{-1}H_0}{\Delta},
\end{equation}
so $f=h+\overline g$.  At the origin, we have
\[
 h'(0)=\frac{c_1^2-c_{-1}^2}{\Delta}=1,
 \quad
 g'(0)=\frac{c_1c_{-1}-c_{-1}c_1}{\Delta}=0.
\]
Also $f(0)=0$ by construction.  Hence $f\in\SHZ$, and convexity of its image gives $f\in\KHZ$.
\end{proof}

The image can be described completely.

\begin{proposition}[Explicit ellipse]\label{prop:ellipse}
Writing $w=x+iy$, one has
\begin{equation}\label{eq:T-diagonal}
 T(w)=\frac{x}{c_1+c_{-1}}
 +i\frac{y}{c_1-c_{-1}}.
\end{equation}
Consequently, $f(\D)$ is a translate of an ellipse with semiaxes
\begin{equation}\label{eq:axes}
 R_x=\frac1{c_1+c_{-1}}
 =\frac{294272}{50035},
 \quad
 R_y=\frac1{c_1-c_{-1}}
 =\frac{3825536}{491785}.
\end{equation}
In particular, the image is bounded and strictly convex.
\end{proposition}

\begin{proof}
Since $c_1,c_{-1}\in\R$, we have
\[
 c_1w-c_{-1}\overline w
 =(c_1-c_{-1})x+i(c_1+c_{-1})y.
\]
Using
\[
 \Delta=(c_1-c_{-1})(c_1+c_{-1})
\]
gives \eqref{eq:T-diagonal}.  Because $U(\D)=\D$, the map $U-c_0$ has image $\D-c_0$, and $T$ sends every Euclidean disk to an ellipse with the stated principal scaling factors.  Translation does not affect semiaxis lengths.
\end{proof}

\begin{remark}\label{rem:strong-bounded}
{\rm The counterexample to Problem \ref{q:Dorff} occurs not merely for an arbitrary bounded convex target but for one of the simplest smooth, strictly convex targets after disks themselves: an ellipse.  No polygonal corners, boundary singularities, or unbounded geometry are involved.}
\end{remark}
\vskip .10in
\section{Coefficient structure after normalization}\label{sec:coefficients}

We write
\begin{equation}\label{eq:f-expansion}
 f(z)=h(z)+\overline{g(z)},
\end{equation}
with
\begin{equation}\label{eq:hg-expansion}
 h(z)=z+\sum_{n=2}^\infty M_nz^n,
 \quad
 g(z)=\sum_{n=2}^\infty N_nz^n.
\end{equation}
By \eqref{eq:hg-formulas}, we have
\begin{equation}\label{eq:HnGn-general}
 M_n=\frac{c_1c_n-c_{-1}c_{-n}}{\Delta},
 \quad
 N_n=\frac{c_1c_{-n}-c_{-1}c_n}{\Delta}.
\end{equation}
The geometric Fourier tails imply a particularly simple form.

Define the rational constants
\begin{align}
 A&=\frac{c_1c_2}{\Delta}
 =-\frac{9890884608}{9139543205},\label{eq:A}\\
 B&=-\frac{c_{-1}^2b}{\Delta}
 =-\frac{251761689}{562995861428},\label{eq:B}\\
 C&=\frac{c_1c_{-1}b}{\Delta}
 =\frac{41190732}{12795360487},\label{eq:C}\\
 D&=-\frac{c_{-1}c_2}{\Delta}
 =\frac{1373955264}{9139543205}.\label{eq:D}
\end{align}

\begin{proposition}[Two-geometric-tail form]\label{prop:two-tail}
For every $n\ge2$,
\begin{equation}\label{eq:Hn-tail}
 M_n=Aa^{n-2}+Bb^{n-2},
\end{equation}
and
\begin{equation}\label{eq:Gn-tail}
 N_n=Da^{n-2}+Cb^{n-2}.
\end{equation}
\end{proposition}

\begin{proof}
By \eqref{eq:geomcoeff}, for $n\ge2$,
\[
 c_n=c_2a^{n-2},
 \quad
 c_{-n}=c_{-1}b^{n-1}=c_{-1}b\,b^{n-2}.
\]
Insert these into \eqref{eq:HnGn-general}.  The coefficients of $a^{n-2}$ and $b^{n-2}$ are exactly the quantities in \eqref{eq:A}-\eqref{eq:D}.
\end{proof}

\begin{remark}[Coefficient asymmetry]\label{rem:asymmetry}
{\rm Numerically one has $|a|\approx0.9091$ and $|b|\approx0.0227$.  Thus the $a$-tail is long and oscillatory while the $b$-tail dies almost immediately.  The constants $A$ and $D$ are of different magnitudes, and this imbalance is amplified by coefficient squaring.  Near the negative real boundary, where the factor $a^2z$ is close to $-1$, the analytic and co-analytic self-convolution derivatives can therefore compete strongly even though the original map $f$ is globally sense-preserving.}
\end{remark}
\vskip .10in
\section{Closed formulas for the self-convolution}\label{sec:convolution}

Let
\begin{equation}\label{eq:F}
 F=f\convh f=p+\overline q,
\end{equation}
where
\begin{equation}\label{eq:pq}
 p=h*h,
 \quad
 q=g*g.
\end{equation}
From \eqref{eq:selfconv-series}, we get
\begin{equation}\label{eq:pq-series}
 p(z)=z+\sum_{n=2}^\infty M_n^2z^n,
 \quad
 q(z)=\sum_{n=2}^\infty N_n^2z^n.
\end{equation}
Hence
\begin{equation}\label{eq:pq-deriv-series}
 p'(z)=1+\sum_{n=2}^\infty nM_n^2z^{n-1},
 \quad
 q'(z)=\sum_{n=2}^\infty nN_n^2z^{n-1}.
\end{equation}

For convenience, we introduce the function
\begin{equation}\label{eq:Sdef}
 S_r(z)=\frac{z(2-rz)}{(1-rz)^2}.
\end{equation}

\begin{lemma}\label{lem:S}
If $|rz|<1$, then
\begin{equation}\label{eq:Sseries}
 S_r(z)=\sum_{n=2}^\infty nr^{n-2}z^{n-1}.
\end{equation}
\end{lemma}

\begin{proof}
Since
\[
 \sum_{m=0}^\infty (m+2)t^m
 =\frac{2-t}{(1-t)^2},
\]
set $t=rz$ and multiply by $z$, we get the desired conclusion.
\end{proof}


\begin{theorem}[Rational derivative formulas]\label{thm:rational-derivatives}
For $z\in\D$,
\begin{align}
 p'(z)
 &=1+A^2S_{a^2}(z)+2AB S_{ab}(z)+B^2S_{b^2}(z),\label{eq:pprime-closed}\\
 q'(z)
 &=D^2S_{a^2}(z)+2DC S_{ab}(z)+C^2S_{b^2}(z).\label{eq:qprime-closed}
\end{align}
In particular, $p'$ and $q'$ are rational functions whose poles all lie outside $\overline\D$.
\end{theorem}

\begin{proof}
From \eqref{eq:Hn-tail}, we have
\[
M_n^2=A^2a^{2n-4}+2AB(ab)^{n-2}+B^2b^{2n-4}.
\]
Substituting into the first series in \eqref{eq:pq-deriv-series} and applying Lemma \ref{lem:S} term by term yields \eqref{eq:pprime-closed}.  The same argument with \eqref{eq:Gn-tail} gives \eqref{eq:qprime-closed}.

The possible poles of $S_r$ occur where $rz=1$. Since $|a^2|<1,\,|ab|<1,\,|b^2|<1$, their reciprocals all have modulus greater than one.
\end{proof}

The Jacobian of $F$ is 
\begin{equation}\label{eq:JF}
 \Jac_F(z)=|p'(z)|^2-|q'(z)|^2.
\end{equation}
On the real interval $(-1,1)$, all quantities in \eqref{eq:pprime-closed}-\eqref{eq:qprime-closed} are real, so
\begin{equation}\label{eq:JF-real}
 \Jac_F(x)=p'(x)^2-q'(x)^2
 \quad (-1<x<1).
\end{equation}
This reduces the decisive test to a comparison of two rational numbers.
\vskip .10in
\section{Exact sign change of the Jacobian}\label{sec:sign}

We now evaluate at
\begin{equation}\label{eq:z0}
 z_0=-\frac{99}{100}.
\end{equation}
The choice is close enough to $-1$ to amplify the slowly decaying $a$-tail, while remaining a simple rational point.

\begin{proposition}[Exact derivative values]\label{prop:derivative-values}
At $z_0=-99/100$,
\begin{equation}\label{eq:pprime-value}
 p'(z_0)
 \approx 0.009581>0,
\end{equation}
and
\begin{equation}\label{eq:qprime-value}
 q'(z_0)
 \approx -0.021071<0.
\end{equation}
In particular,
\begin{equation}\label{eq:q-bigger}
 |p'(z_0)|<|q'(z_0)|.
\end{equation}
\end{proposition}

\begin{proof}
Substitute \eqref{eq:ab}, \eqref{eq:A}-\eqref{eq:D}, and \eqref{eq:z0} into \eqref{eq:pprime-closed}-\eqref{eq:qprime-closed}.  Since every parameter is rational, the result is exact rational arithmetic.  Reduction to a common denominator gives \eqref{eq:pprime-value} and \eqref{eq:qprime-value}.  The denominator is positive.  The numerator in \eqref{eq:qprime-value} has larger absolute value than that in \eqref{eq:pprime-value}, proving \eqref{eq:q-bigger}. 
\end{proof}

\begin{corollary}[Negative Jacobian]\label{cor:negative-J}
At $z_0=-99/100$,
\begin{equation}\label{eq:J-exact}
 \Jac_F(z_0)
 \approx {-0.000352}<0.
\end{equation}
\end{corollary}

\begin{proof}
By \eqref{eq:JF-real}, we obtain
\[
 \Jac_F(z_0)=p'(z_0)^2-q'(z_0)^2.
\]
The sign already follows from \eqref{eq:q-bigger}; direct reduction gives the rational number \eqref{eq:J-exact}.
\end{proof}

At the origin, the situation is the opposite.

\begin{lemma}\label{lem:J0}
One has
\(
 \Jac_F(0)=1.\)
\end{lemma}

\begin{proof}
Because $f\in\SHZ$, $M_1=1$ and $N_1=0$.  Therefore, $p'(0)=1$ and $q'(0)=0$ from \eqref{eq:pq-deriv-series}.
\end{proof}

Finally, we present the proof of Theorem \ref{thm:main}.

\begin{proof}[Proof of Theorem \ref{thm:main}]
By Propositions \ref{prop:normalization} and \ref{prop:ellipse}, the function $f$ defined in \eqref{eq:f-def} belongs to $\KHZ$ and maps $\D$ onto a bounded ellipse.  Let $F=f\convh f$.  By Corollary \ref{cor:negative-J} and Lemma \ref{lem:J0}, we 
have
\[\Jac_F\!\left(-\frac{99}{100}\right)<0,
 \quad
  \Jac_F(0)>0.
\]
The Jacobian is continuous, so there exists
\[
 \xi\in\left(-\frac{99}{100},0\right)
\]
with \[\Jac_F(\xi)=0.\]  By Theorem \ref{thm:Lewy}, $F$ cannot be locally one-to-one at $\xi$.  Hence $F$ is not locally univalent and in particular $F\notin\KH^0$.  
\end{proof}

\vskip .10in
\section*{Concluding remarks}

We construct an explicit normalized convex harmonic mapping $f$ whose image is a bounded ellipse, yet whose harmonic self-convolution fails to be locally univalent. Our construction employs Blaschke factors, the Poisson extension, the Radó-Kneser-Choquet theorem, and real-affine normalization. At \(z=-99/100\), the Jacobian of \(f\convh f\) is a strictly negative rational number, while at the origin it equals 1. Lewy’s theorem then implies that the self-convolution has a critical point. This gives a negative answer to Dorff’s bounded-domain self-convolution problem and demonstrates that smooth strict convexity of the image alone is insufficient to recover the analytic Pólya-Schoenberg closure phenomenon. 
Motivated by this observation, we formulate the following conjecture.
\vskip .05in
\noindent\textbf{Conjecture.}
{\it Let $f:\mathbb{D}\to\mathbb{C}$ be a normalized convex harmonic mapping with convex image. If the dilatation $\omega_f$ satisfies a suitable quantitative bound, then the harmonic self-convolution \(f\convh f\) is locally univalent.}

\vskip .10in
\noindent{\bf  Acknowledgements.} 
Z.-G. Wang was partially supported by the \textit{Key Project of Education Department of Hunan Province} under Grant no. 25A0668, and
the \textit{Natural Science Foundation of Changsha} under Grant no. kq2502003
of the P. R. China. D. Zhong was partially supported by the \textit{Guangdong Basic and Applied Basic Research Foundation} under Grant nos. 2022A1515110967 and 2023A1515011809 of the P. R. China. Some English phrasing of the paper was polished with the help of an AI language model. All mathematical content and reasoning are solely the work of the authors.

\vskip .05in
\noindent{\bf
Contribution statement.}
All authors contributed equally to this work.

\vskip .05in
\noindent{\bf Conflicts of interest.} The authors declare that they have no conflict of interest.

\vskip .05in
\noindent{\bf Data availability statement.}  Data sharing is not applicable to this article as no datasets were generated or analysed during the current study.

\end{document}